\documentclass[a4paper,reqno,12pt]{amsart}
\usepackage{amsmath}
\usepackage[hidelinks]{hyperref}
\usepackage{amssymb}
\usepackage{ifthen}
\usepackage{graphicx}
\usepackage{comment}
\usepackage{float}
\usepackage{cite}
\usepackage{amsfonts}
\usepackage{amscd}
\usepackage{amsxtra}
\usepackage{color}
\usepackage{enumerate}
\usepackage{enumitem}

\numberwithin{equation}{section}
\newtheorem{theorem}{Theorem}[section]
\newtheorem{lemma}[theorem]{Lemma}
\newtheorem{corollary}[theorem]{Corollary}
\newtheorem{proposition}[theorem]{Proposition}

\newtheorem{definition}[theorem]{Definition}

\theoremstyle{definition}
\newtheorem{remark}[theorem]{Remark}

\numberwithin{equation}{section}

\newcounter{alphabet}
\newcounter{tmp}
\newenvironment{Thm}[1][]{\refstepcounter{alphabet}%
	\bigskip%
	\noindent%
	{\bf Theorem \Alph{alphabet}}%
	\ifthenelse{\equal{#1}{}}{}{ (#1)}%
	{\bf .} \itshape}{\vskip 8pt}



\begin{document}
	\author[Vibhuti Arora]{Vibhuti Arora}
	\address{Vibhuti Arora, Department of Mathematics, National Institute of Technology Calicut, 673 601, India.}
	\email{vibhutiarora1991@gmail.com, vibhuti@nitc.ac.in}
	
	\author[Vinayak M.]{Vinayak M.}
	\address{Vinayak M., Department of Mathematics, National Institute of Technology Calicut, 673 601, India.}
	\email{mvinayak.math@gmail.com, vinayak\_p220286ma@nitc.ac.in}

	\begin{abstract}
		This paper introduces the second Bohr radius for vector-valued holomorphic functions defined on arbitrary complete Reinhardt domains. We aim to establish the lower and upper bounds of the second Bohr radius in both finite and infinite-dimensional settings. 
		Additionally, we provide specific estimates that connect the Second Bohr radius to a symmetric Banach space. We also explore the relationships between our findings and certain existing results.
	\end{abstract}
	\title[Estimates of the Second Bohr radius]
	{Estimates of the Second Bohr radius for vector-valued Holomorphic functions}

	\newcommand{\D}{{\mathbb D}}
	\newcommand{\C}{{\mathbb C}}
	\newcommand{\real}{{\operatorname{Re}\,}}
	\newcommand{\Log}{{\operatorname{Log}\,}}
	\newcommand{\Arg}{{\operatorname{Arg}\,}}
	\newcommand{\ds}{\displaystyle}
	

	\maketitle
	\section{Introduction}
	In 1914, Harald Bohr \cite{B014} started studying the absolute convergence of the well-known Dirichlet series and linked it with the theory of analytic function $f$ defined on the unit disk $\mathbb{D}:=\{z \in \mathbb{C}: |z|<1 \}$, such that $|f(z)|\leq 1$ and having the series expansion
	\begin{equation*}
		f(z) = \sum\limits_{k=0}^{\infty} a_k z^k. 
	\end{equation*}
	Bohr initially proved the following inequality 
	\begin{equation}\label{E3.1.2}
		\sum\limits_{k=0}^{\infty} |a_k|r^k \leq 1, 
	\end{equation}
	for $r \leq 1/6$, and it was not sharp. Later, Riesz, Schur, and Wiener independently obtained for $|z|=r\leq 1/3$. \eqref{E3.1.2} is called the {\em Bohr inequality} and the largest radius $1/3$ which holds the inequality \eqref{E3.1.2} is called the {\em Bohr radius}.
	
	\medskip
	
	Researchers in Banach algebra have discovered an intriguing link between operator algebra and the famous von Neumann inequality.
	By considering the Banach algebra $X_\beta=l^1(\mathbb{N}), ~ 0<\beta \leq 1/3$, of all summable sequences with the norm $\|x\|:=\beta^{-1}\|x\|_1$, Dixon \cite{DI95} showed that $X_\beta$ is not an operator algebra, but is a non-unital Banach algebra which satisfies the von Neumann inequality. The result was obtained using the Bohr inequality, which opened a new path for many mathematicians to explore local Banach space theory. Later,  Paulsen et al. \cite{PA02} extended the theory of the Bohr inequality to the Banach algebra of bounded analytic functions, and certain multi-dimensional domains. Investigating the Bohr inequality for certain operators was another area of interest in this field. In particular, the study of the Bohr inequality connected with several integral operators can be seen in \cite{KA19, SH21}. A comprehensive survey of recent trends in this area over the past decades can be found in the article \cite{MU16}. For insights into other recent works in this area, see \cite{VI24, AR23, BH18, KA17, AL19, KPW25} and the references cited within them.
	
	\medskip
	
	The notion of power series can be extended to the multidimensional space $$\mathbb{C}^n= \{(z_1,\dots,z_n): z_k \in \mathbb{C}, ~ k=1,\dots,n\}, ~ n>1$$ in a natural way. 
	Before going to the case, we need some prerequisites. Throughout the discussion, we use the multi-index notation $\alpha= (\alpha_1, \dots, \alpha_n), ~\alpha_k \in \mathbb{N}_0:=\mathbb{N} \cup \{0\}, ~ k=1, \dots,n$ and $\bar{0} \in \mathbb{C}^n$ means $(0,\dots,0)$. The operation $z^\alpha$, for $z \in \mathbb{C}^n$ is defined by $ z^\alpha := z_1^{\alpha_1} \cdots z_n^{\alpha_n}.$ Also, $|\alpha|:= \alpha_1+ \cdots +\alpha_n$, and $\alpha ! := \alpha_1 ! \cdots \alpha_n !$. A domain $\Omega$ is said to be {\em Reinhardt} if for any $z=(z_1, \dots, z_n) \in \Omega$, then we have $(\lambda_1 z_1, \dots, \lambda_nz_n) \in \Omega$, where $|\lambda_k|=1, ~ k=1, \dots, n.$ The domain $\Omega$ is said to be {\em complete Reinhardt} if for any $z=(z_1, \dots, z_n) \in \Omega$, then we have $(\lambda_1 z_1, \dots ,\lambda_nz_n) \in \Omega$, where $|\lambda_k|\leq 1, ~ k=1, \dots ,n.$ 
	
	\medskip
	
	Now we recall the concept of power series about $\bar{0} \in \mathbb{C}^n$, of a Banach space-valued function $f$; that is, the series of the form
	\begin{equation}\label{E3.1.3}
		f(z) : = \sum_{\alpha \in {\mathbb{N}_0}^n} a_\alpha z^\alpha = \sum_{\alpha \in {\mathbb{N}_0}^n} a_{(\alpha_1,\dots, \alpha_n)} z_1^{\alpha_1} \cdots z_n^{\alpha_n},
	\end{equation}
	where $a_\alpha \in X$, $X$ is an arbitrary Banach space.
	The notion of Fr\'echet differentiability generalizes the idea of classical differentiation to any normed linear space. For two normed spaces $X$ and $Y$ over $\mathbb{C}$ and a function $f:A \to Y$, where $A$ open in $X$, we say $f$ is {\em Fr\'echet differentiable} at $x \in A$ if there exists $T_x \in BL(X,Y)$ (the space of all bounded linear functions from $X$ to $Y$) satisfying the condition
	
	\begin{equation*}
		\dfrac{f(x+k)-f(x)-T_x(k)}{\|k\|} \to 0 ~ \text{as} ~ k \to 0.
	\end{equation*}
	Here $T_x$ is called the {\em differential} of $f$ at the point $x$. If $f$ is Fr\'echet differentiable at each point in $A$, then we say $f$ is {\em holomorphic on A}. See \cite[p. 354]{DE19} for a detailed discussion on this topic. 
	
	\medskip
	
	Analogous to the classical result by Bohr in single variable, Boas and Khavinson \cite{BO97} started discussing the Bohr radius $K_n$ for the complex-valued holomorphic functions defined on polydisk $\mathbb{D}^n :=\{(z_1, \dots z_n) \in \mathbb{C}^n : |z_k|<1, ~ k=1, \dots, n\},$ and obtained the following bounds for the case $n>1$:
	\begin{equation*}
		\dfrac{1}{3\sqrt{n}}<K_n<2 \cdot \sqrt{\dfrac{\log n}{n}}.
	\end{equation*}
	Later, Boas \cite{BO00} generalized this discussion to the unit ball $B_p^n : =\{z=(z_1, \dots z_n) \in \mathbb{C}^n : \|z\|_p<1\}$, where
	\begin{align*}
		\|z\|_p := \begin{cases}
			\left(\sum\limits_{k=1}^n |z_k|^p\right)^{1/p}, & ~ 1 \leq p < \infty, \\
			\max\{|z_k|, ~ k=1, \dots, n\}, & ~ p=\infty.
		\end{cases}
	\end{align*}

	\noindent But so far, apart from the bounds for the $n$-dimensional Bohr radius, no exact value has been obtained. Recently, Blasco extended the discussion of the Bohr radius to any arbitrary Banach space \cite{BL10}. Current research trends regarding these area can be found in \cite{AR23, DA20, BA14, SH23, DE18, DE219, LI21, PO19, SH123, DE11, AI08, AKP25, DE06, JI25, LI23}.
	
	\medskip

	Aizenberg \cite{AI00} introduced a new concept called the {\em second Bohr radius} for a bounded complete Reinhardt domain $\Omega \subseteq \mathbb{C}^n$, which is denoted by $B_n(\Omega)$. For an analytic function $f:\Omega \to \mathbb{D}$ with the series expansion \eqref{E3.1.3}, $B_n(\Omega)$ is the largest radius $r$ such that 
	\begin{equation*}
		\sum_{\alpha \in {\mathbb{N}_0}^n} \sup_{r\Omega} |c_\alpha z^\alpha| <1,
	\end{equation*}
	where $r\Omega:= \{(rz_1,\dots,rz_n) : (z_1,\dots,z_n) \in \Omega\}$. Note that since $$\sup \limits _{r\Omega} \sum_{\alpha \in {\mathbb{N}_0}^n}|a_\alpha z^\alpha| \leq \sum_{\alpha \in {\mathbb{N}_0}^n} \sup_{r\Omega}|a_\alpha z^\alpha|,$$ the second Bohr radius gives a lower bound for the classical Bohr radius in multidimensional case. Later, Boas \cite{BO00} derived this concept to the unit ball using any $p$ norm, which gives the estimates for the second Bohr radius in the context of $\Omega=B_p^n$ and obtained the following result.
	
	\begin{Thm}
		For $n>1$, the second Bohr radius $B_n(B_p^n)$ satisfies the following bounds:
		\begin{enumerate}[label=(\alph*)]
			\item For $1\leq p \leq 2,$
			\begin{equation*}
				\dfrac{1}{3n}<  1-\sqrt[n]{\dfrac{2}{3}} \leq B_n(B_p^n) <4 \left(\dfrac{\log n}{n}\right).
			\end{equation*}
			\item For $2\leq p \leq \infty,$
			\begin{equation*}
				\dfrac{1}{3} \left(\dfrac{1}{n}\right)^{1/2+1/p}\leq B_n(B_p^n) <4 \left(\dfrac{\log n}{n}\right)^{1/2+1/p}.
			\end{equation*}
		\end{enumerate}
	\end{Thm}
	In 2004, Defant et al. \cite{DE04} estimated the second Bohr radius for any complete Reinhardt domain, resulting in the following conclusion.
	
	\begin{Thm}
		Let $\Omega$ be a complete Reinhardt domain in $\mathbb{C}^n$. Then we have
		
		\begin{equation*}
			\dfrac{1}{3} \max \left(\dfrac{1}{n}, \dfrac{1}{\sqrt{n} S(\Omega, B_\infty^n) S(B_\infty^n,\Omega)}\right) \leq B_n(\Omega) \leq 2^{3/2} \sqrt{\log n}  \cdot e^3 \dfrac{S(\Omega, B_2^n)}{n},
		\end{equation*}
		where $S(\Omega_1, \Omega_2)$ is defined as
		\begin{equation}\label{D3.1.3}
			S(\Omega_1, \Omega_2) := \inf \{\beta>0 : \Omega_1 \subset \beta \cdot\Omega_2\},
		\end{equation}
		for any two complete Reinhardt domains $\Omega_1$ and $\Omega_2$.
	\end{Thm}
	Defant et al. \cite{DE12} recently introduced the theory of $\lambda$-Bohr radius for vector-valued holomorphic functions on  $\mathbb{D}^n$, defined as follows: Let $T: X \to Y$ be a bounded linear operator with $\| T \rVert \leq \lambda$ and $f: \mathbb{D}^n \to X$ be a holomorphic function with the series expansion \eqref{E3.1.3}. The {\em $\lambda$-Bohr radius of $T$}, denoted by $K(\mathbb{D}^n, T, \lambda)$ is the largest $r\geq 0$ such that the inequality
	\begin{equation*}
		\sup _{ z \in r\mathbb{D}^n} \sum_{\alpha \in \mathbb{N}_0^n}  \| T(a_\alpha)z^\alpha\rVert_Y \leq \lambda \sup_{z \in \mathbb{D}^n} \|f(z)\|_X=\lambda \|f\|_\infty.
	\end{equation*}
	It is to be noted that the same definition can be extended to an arbitrary bounded complete Reinhardt domain in $\mathbb{D}^n$. Recently, Kumar et al. \cite{SH23} extended this to the functions defined on $B_p^n$ for any $p \in [1,\infty]$. We aim to define the analogous concept of the $\lambda$-second Bohr radius for vector-valued holomorphic functions.

	\medskip
	
	Das \cite{DA23} has recently established a logarithmic lower bound for the second Bohr radius of functions defined on $B_p^n$. To the best of our knowledge, except for the works in \cite{BO00, AI00, DA23, DE04}, the estimation of the second Bohr radius has not been studied. Our main objective is to define the analogous version of the second Bohr radius, similar to the vector-valued Bohr radius defined in \cite{DE12}.
	
	\medskip

	The paper is organized as follows. In Section 2, we will first define the $\lambda$-second Bohr radius for vector-valued holomorphic functions in several variables and present our main results. After that, in Section 3, we present some preliminary results that will help us to prove our main theorems. Section 4 is dedicated to providing proofs of our main results.
	
	\medskip
	\section{Definitions and Main Results}
	We start this section with the definition of the $\lambda$-second Bohr radius. We assume that $X$ and $Y$ denote two arbitrary Banach spaces defined over a complex field. Furthermore, we denote $\Omega$ as a bounded complete Reinhardt domain in $\mathbb{C}^n$.
	\begin{definition}\label{D1}
		Let $T: X \to Y$ be a non-null bounded linear operator with $\| T \rVert \leq \lambda$. The $\lambda$-second Bohr radius of $T$, denoted by $B(\Omega, T, \lambda)$ is the largest $r\geq 0$ such that the inequality
		\begin{equation}\label{E3.2.1}
			\sum_{\alpha \in \mathbb{N}_0^n} \sup _{z \in r \cdot \Omega} \| T(a_\alpha)z^\alpha\rVert_Y \leq \lambda \sup_{z \in \Omega}\| f(z) \rVert_X
		\end{equation}
		holds for all $X$-valued analytic functions $f$ defined on $\Omega$ with the series expansion \eqref{E3.1.3}. We denote $B(\Omega,T,\lambda)=:B(\Omega,X,\lambda)$, if $T$ is the identity on $X$.
	\end{definition}
	\begin{remark}
		If $f$ is an unbounded function, then the inequality \eqref{E3.2.1} will hold trivially. So throughout the discussion, we assume $f$ to be a bounded analytic function. Also note that  Definition \ref{D1} of $B(\Omega,\mathbb{C},1)$ coincides with the second Bohr radius defined in \cite{AI00}.
	\end{remark}
	It is to be noted that every analytic function in several variables is closely connected with homogeneous polynomials. More precisely, suppose that $f$ is of the form \eqref{E3.1.3}. Then $f$ can also be expressed as
	$$f(z) = \sum_{m=0}^\infty \sum_{|\alpha|=m}  a_\alpha z^\alpha = \sum_{m=0}^\infty P_m(z),$$
	where $P_m$ is the $m$-homogeneous polynomial given by
	\begin{equation}\label{E3.2.2P}
		P_m(z) = \sum_{|\alpha|=m}  a_\alpha z^\alpha.
	\end{equation}
	The study of Bohr radius for the space of $m$-homogeneous polynomials and its relation with the Bohr radius for arbitrary analytic functions can be seen in \cite{DE03, DE12}.
	We now define the second Bohr radius $B_m(\Omega, T, \lambda)$ for $m$-homogeneous polynomials as follows.
	\begin{definition}
		Let $T: X \to Y$ be a non-null bounded linear operator with $\| T \rVert \leq \lambda$. The $\lambda$-second Bohr radius for $m$-homogeneous polynomial, denoted by $B_m(\Omega, T, \lambda)$ is the largest $r\geq 0$ such that the inequality
		\begin{equation*}
			\sum_{|\alpha|=m} \sup _{z \in r \cdot \Omega} \| T(a_\alpha)z^\alpha\rVert_Y \leq \lambda \sup_{z \in \Omega}\| P_m(z) \rVert_X
		\end{equation*}
		holds for all $X$-valued $m$-homogeneous polynomials $P_m$ defined on $\Omega$, which has the representation \eqref{E3.2.2P}. We denote $B_m(\Omega,T,\lambda)=B_m(\Omega,X,\lambda)$, if $T$ is the identity on $X$. 
		
	\end{definition}
	
	\begin{remark}\label{R3.2.1} 
		Note that similar notation as in Table 1 can be provided for $B_m(\Omega, T,\lambda)$ under the same conditions. Now, observe that from the definition itself, we have the relation 
		$$\inf_{m \in \mathbb{N}}B_m(\Omega, T, \lambda) \geq B(\Omega, T, \lambda).$$ 
		Also note that $K(\Omega, T, \lambda) \geq B(\Omega, T, \lambda),$ due to the fact that 
		$$ \sup _{z \in r \cdot \Omega} \sum_{\alpha \in \mathbb{N}_0^n} \| T(a_\alpha)z^\alpha\rVert_Y  \leq  \sum_{\alpha \in \mathbb{N}_0^n} \sup _{z \in r \cdot \Omega} \| T(a_\alpha)z^\alpha\rVert_Y.$$  In particular, suppose we take $\Omega = \mathbb{D}^n$. Then due to the fact that for each $k$, $z_k$ simultaneously maximizes the value of $|z^\alpha|$ for all indices $\alpha$, we have 
		$$ \sup _{\| z \rVert_\infty <r} \sum_{\alpha \in \mathbb{N}_0^n} \| T(a_\alpha)z^\alpha\rVert_Y  =  \sum_{\alpha \in \mathbb{N}_0^n} \sup _{\| z \rVert_\infty <r} \| T(a_\alpha)z^\alpha\rVert_Y. $$
		This concludes the following result.
	\end{remark}
	
	\begin{proposition}\label{3P1}
		The first and second $\lambda$-Bohr radii of the unit polydisk in $\mathbb{C}^n$ are the same. That is, 
		$K(\mathbb{D}^n, T, \lambda) = B(\mathbb{D}^n, T, \lambda).$ In particular, $B(\mathbb{D},\mathbb{C},1)=1/3$.
	\end{proposition}
	In \cite[Theorem 1.2]{BL10}, Blasco showed that the classical Bohr radius for the function which takes values on $B_p^m$ is zero for $m \geq 2, ~1 \leq p \leq \infty$. Here we will give an analogous result for the $\lambda$-second Bohr radius for functions with domain $\mathbb{D}^n$. The result is as follows.
	\begin{theorem}\label{3Tm1}
		For $1 \leq p \leq \infty$, $n\geq 1, ~ \text{and} ~  m\geq 2$, let $f:\mathbb{D}^n \to \overline{B_p^m}$ be an analytic function with the series expansion \eqref{E3.1.3}. Then $B(\mathbb{D}^n, \overline{B_p^m},1) =0$.
	\end{theorem}

	\begin{remark}
		Note that by Proposition \ref{3P1}, we have $B(\mathbb{D}^n,X,1)=K(\mathbb{D}^n,X,1)$, for any Banach space $X$. Then by Theorem \ref{3Tm1}, we directly get that the classical vector-valued Bohr radius for functions from $\mathbb{D}^n$ to $\overline{B_p^m}$, which is denoted by  $K(\mathbb{D}^n,\overline{B_p^m},1)=0, ~ 1 \leq p \leq \infty$. This shows the importance of considering the extra term $\lambda$ in the RHS of inequality \eqref{E3.2.1} in the definition of the $\lambda$-second Bohr radius. Moreover,  we get $K(\mathbb{D},\mathbb{C}^m,1)=0$ (by changing the role of $z_k$ to the variable $z \in \mathbb{D}$ in the proof), which coincides with Theorem 1.2 of \cite{BL10}.
	\end{remark}
	
	Now, we state the following result, which gives a lower bound for the $\lambda$-second Bohr radius. Moreover, the following result shows that for every $\lambda>1$ such that $\|T\|<\lambda$, the $\lambda$-second Bohr radius of $T$ is always positive.
	\begin{theorem}\label{3T1}
		For $1 \leq p \leq \infty$, let $f:B_p^n \to X$ be an analytic function with the series expansion \eqref{E3.1.3}. Then we have 
		$$B(B_p^n,T,\lambda) \geq \frac{C}{n^{1+1/p}},$$
		where $C$ is a positive number defined by
		\begin{equation*}
			C= 
			\begin{cases}
				\max (\rho, \sigma_1), & \text{if } ~\|T\| \geq 1,\\
				\max (\rho, \sigma_2), & \text{if } ~0<\|T\|<1,
			\end{cases}
		\end{equation*}
		$$\rho=\dfrac{\lambda-\|T\|}{2\lambda-\|T\|}, ~ 
		\sigma_1 = \dfrac{\lambda-\|T\|}{(\lambda-\|T\|+1)\|T\|}, ~ \text{and } \sigma_2 = \dfrac{\lambda-\|T\|}{(\lambda-\|T\|+1)}. $$  
	\end{theorem}
	
	\medskip
	
	We conclude this section by giving the following result, by considering a particular case of Theorem \ref{3T1}.  Moreover, if $\lambda=1$, then we only have an obvious lower bound that $B(\Omega,T,\lambda) \geq 0.$
	\begin{corollary}\label{C31}
		Let $f:B_p^n \to X$ be an analytic function with the series expansion \eqref{E3.1.3}. Then for $1\leq p \leq \infty$ and for any $\lambda>1$, the $\lambda$-second Bohr radius satisfies
		$$B(B_p^n,X, \lambda) \geq \frac{\lambda-1}{\lambda \cdot n^{1+1/p}}.$$
	\end{corollary}
	\begin{proof}
		Since $T$ is the identity operator, $\|T\|=1$. So from Theorem \ref{3T1}, we have 
		$$C= \max \left(\dfrac{\lambda-1}{2\lambda-1},\dfrac{\lambda-1}{\lambda}\right) = \dfrac{\lambda-1}{\lambda},$$
		which gives the required lower bound.
	\end{proof}
	
	Next, we provide a logarithmic lower bound in a different setting. Prior to that, we recall some functional analytic concepts.  A partially ordered set $(A,\leq)$ is said to be a {\em lattice} if any $\{x,y\} \subseteq A$ has a least upper bound (denoted by $x \vee y$) and a greatest lower bound (denoted by $x \wedge y$). We give the notation $|x|=x \vee (-x)$, for any $x \in A$. A vector space with a lattice structure is called a {\em vector lattice}. A Banach space $X$ with a lattice structure that satisfies the condition $|x| \leq |y|$ implies $\|x\| \leq \|y\|$, for any $x, y \in X$ is called a {\em Banach lattice}. For $q \in [1,\infty)$, we define the {\em $q$-concave Banach lattice} as a Banach lattice $X$ satisfying the condition that there exists a constant $C>0$ such that for any $x_1,\dots,x_n \in X$, we have
	\begin{equation*}
		\Bigg(\sum_{k=1}^n \|x_k\|^q\Bigg)^{1/q} \leq C \Bigg\|\Bigg(\sum_{k=1}^n |x_k|^q\Bigg)^{1/q}\Bigg\|,
	\end{equation*}
	where the right-hand side is defined as 
	\begin{equation*}
		\Bigg(\sum_{k=1}^n |x_k|^q\Bigg)^{1/q}:={\Large \wedge}  \Bigg\{\sum_{k=1}^n a_kx_k: a_k \in \mathbb{C}, \sum_{k=1}^n |a_k|^q<1 \Bigg\}.
	\end{equation*}
	Note that the sequence space $l_q$ is $q'$-concave, where $q':=\max (q,2)$. For $p,q \in [1,\infty)$, an operator $T:X\to Y$ is said to be $(p,q)$-summing if there exists a constant $C>0$ such that for any $x_1,\dots,x_n \in X$, we have
	\begin{equation*}
		\Bigg(\sum_{k=1}^n \|T(x_k)\|^p\Bigg)^{1/p} \leq \sup_{\phi \in X^*, \|\phi\|<1} \Bigg(\sum_{k=1}^n |\phi(x_k)|^q\Bigg)^{1/q}.
	\end{equation*}
	See \cite{DE12, LI79} for a detailed study of these
	topics.
	
	\medskip
	Now we are in a stage of stating our result on a logarithmic lower bound for vector-valued analytic functions. We use a similar idea as in \cite{DA23} to obtain our result.
	\begin{theorem}\label{LOGTHM}
		Let $Y$ be a $q$-concave Banach lattice, for $q \in [2, \infty)$ and $T:X \to Y$ be an $(r,1)$-summing operator such that $\|T\| \leq \lambda$, where $r \in [1,q)$.  If $f:B_p^n \to X$ be an analytic function with the series expansion \eqref{E3.1.3}, and $\|f\|<1$, then there exists a constant $C_0>0$ such that 
		\begin{equation*}
			B(B_p^n,T,\lambda) \geq \dfrac{C_0}{1+2\lambda} \Bigg(\dfrac{\|T\|-\lambda}{\|T\|-2\lambda}\Bigg) \dfrac{(\log n)^{1-\frac{1}{q}}}{n^{\frac{pq-p+q}{pq}}}.
		\end{equation*}
	\end{theorem}
	
	Next, we discuss certain upper bounds for the $\lambda$-second Bohr radius. The following result gives immediate lower and upper bounds in the case of finite-dimensional Banach space-valued functions, which is a direct consequence of Lemma \ref{L3.2.2}. 
	\begin{theorem}\label{THM2.12}
		Let $X$ be a finite-dimensional Banach space. For $1 \leq p \leq \infty$, let $f: B_p^n \to X$ be an analytic function with series expansion \eqref{E3.1.3}. Then there exists a universal constant $E,F_X>0$ such that 
		$$\dfrac{F_X (\lambda-1) n^{1/p-1/2} \sqrt{\log n}}{2 \lambda-1}\leq B(B_p^n,X,\lambda) \leq E \lambda^2 (n)^{1/p -1}\sqrt{\log n} ~;$$
		here $E$ is a universal constant and $F_X$ is a constant depending on the space $X$. 
	\end{theorem}   
	
	\medskip
	
	Recall that a {\em Schauder basis} for a vector space $X$ is a sequence $(e_k)$ of elements of $X$ such that for every $x \in X$, there exist unique scalars $c_k \in \mathbb{C}$ such that $x = \sum_{k=1}^\infty c_k e_k$. The sequence space $l^p, ~ 1 \leq p <\infty$ has a Schauder basis, but $l^\infty$ does not have a Schauder basis. In addition, it is trivial that a sequence space can identify every vector space with a Schauder basis with every element $x=\sum_{k=1}^\infty c_k e_k\in X$ can be considered as the unique sequence $(c_k)$. We now present our result, which provides an upper bound for $B(B_p^n, T, \lambda)$ in connection with the Schauder basis. 
	\begin{theorem}\label{T3.4.3}
		Let $X$ be a Banach space having a Schauder basis. For $1 \leq p < \infty$ and $n>1$, let $f: B_p^n \to X$ be an analytic function with the series expansion \eqref{E3.1.3}. Then the $\lambda$-second Bohr radius satisfies
		$$B(B_p^n,X,\lambda) \leq  \begin{cases}
			\lambda e^3 2^{3/2}\dfrac{(\log n)^{\frac{1}{2} + \frac{1}{p}}}{n}   , ~ \text{if}~ 1 \leq p \leq 2,\\
			\lambda e^3 2^{3/2}\left( \dfrac{\log n}{n}\right)^{\frac{1}{2} + \frac{1}{p}}   , ~ \text{if} ~ 2 \leq p < \infty.
		\end{cases}$$
	\end{theorem}
	\begin{remark}
		In Theorem \ref{T3.4.3}, if we consider $X=\mathbb{C}$ with  $\lambda=1$ and $|f(z)| \leq 1,$ then the theorem reduces to Example 3.6 of \cite{DE04}, for the case such that $p_k=p, ~ k=1, \dots,n.$
	\end{remark}
	
	In the previous results, we studied the $\lambda$-second Bohr radius associated with unit balls with $p$-norms. Note that the unit vector basis $(e_k)$ of $l_p$ spaces has the property that every permutation of $(e_k)$ is equivalent and each one forms a basis for $l_p$ space. The theory of symmetric and unconditional bases generalizes this property to arbitrary Banach spaces. A Schauder basis $(x_k)$ in a Banach space $X$ is said to be {\em unconditional} if there exists a constant $K \geq 0$ such that 
	\begin{equation*}
		\left\|\sum_{k=1}^\infty \zeta_k a_k x_k\right\|_{X} \leq K \left\|\sum_{k=1}^\infty a_k x_k\right\|_{X},
	\end{equation*}
	for all $\zeta_k, a_k \in \mathbb{C}$ such that $|\zeta_k| \leq 1.$ The best such constant $K$ is called {\em unconditional basis constant of $(x_k)$}, which is denoted by $\chi(x_k)$. We denote by $\chi_M(\mathcal{P}(^mX))$, for the unconditional basis constant of $(z^\alpha)$ for the space $\mathcal{P}(^mX)$ of $m$-homogeneous polynomials $P$ in the space $X$ occupied with the norm
	\begin{equation*}
		\|P\|_{\mathcal{P}(^mX)} := \sup \{\|P(z)\| : z \in X, \|z\|\leq 1\}.
	\end{equation*}
	
	\medskip
	
	\noindent A basis $(x_k)$  is said to be {\em symmetric} if for each permutation $\sigma$ of the integers, we have
	\begin{enumerate}[label=(\alph*)]
		\item the sequence $(x_{\sigma(k)})$ also forms a basis for $X$,
		\item the sequences $(x_k)$ and $(x_{\sigma(k)})$ are equivalent; which means a series $\sum_{k=1}^\infty a_k x_k$ converges if and only if $\sum_{k=1}^\infty a_k x_{\sigma(k)}$ converges.
	\end{enumerate}
	A Banach space with a symmetric basis is called a {\em symmetric Banach space}. It is to be noted that every symmetric Banach space has an unconditional basis. For a detailed discussion on these topics related to Banach spaces, see  \cite{LI77, LI79}.
	
	\medskip
	
	For the upcoming result, we denote the unit ball in an arbitrary Banach space $W$ as $B_{W}$. Additionally, we use the notation $W^*$ for the dual space of $W$. We are now prepared to present our final result, which connects the second Bohr radius to the duality of symmetric Banach spaces. The result is outlined as follows.
	
	\begin{theorem}\label{THM2.15}
		Let $W_n =(\mathbb{C}^n, \| \cdot \|)$ be a symmetric Banach space such that $\chi((e_k))=1$, where $(e_k)$ denotes the canonical basis for $W_n$. Define the number $b_n(\lambda)$ as  $$b_n(\lambda) :=  B(B_{W_n}, X,\lambda)  B(B_{W_n^*},X, \lambda).$$
		Then 
		\begin{equation*}
			b_n(\lambda) \leq \dfrac{(8 \lambda^2 e^6 \log n) d(W_n,l_2^n)}{n},
		\end{equation*}
		where $d(W_n,l_2^n)$ is the Banach-Mazur distance.
		Consequently, $ \lim \limits_{n \to \infty} b_n(\lambda) =0.$
	\end{theorem}
	
	When $\lambda=1$ and $X=\mathbb{C}$,   with $|f(z)| <1$, Theorem \ref{THM2.15} reduces to Corollary 3.11 of \cite{DE04}.
	
	\section{Preliminary Results}
	This section is dedicated to certain preliminary results that are very necessary in proving our main results. First, we deduce an estimate for the $\lambda$-second Bohr radius. Our Lemma is an extension to the vector-valued case of  Lemma 3.1 in \cite{DE04}, which deals with the Bohr radius for complex-valued functions.
	\begin{lemma}\label{L3.2.2}
		Let $\Omega_1$ and $\Omega_2$ be two bounded complete Reinhardt domains in $\mathbb{C}^n$. Then we have
		\begin{enumerate}[label=(\alph*)]
			\item If $S(\Omega_1, \Omega_2)$ is as in \eqref{D3.1.3}, then
			$$B(\Omega_2,T,\lambda)/ [S(\Omega_1, \Omega_2) \cdot S(\Omega_2, \Omega_1)]\leq B(\Omega_1,T,\lambda) \leq S(\Omega_1, \Omega_2) \cdot S(\Omega_2, \Omega_1) \cdot B(\Omega_2,T,\lambda).$$ 
			\item For any $\rho>0$, $B(\Omega_1,T,\lambda)=B(\rho \cdot \Omega_1,T,\lambda)$.
			\item If $\Omega_2 \subset \Omega_1 \subset \rho \cdot \Omega_2$, where $\rho>0$, then we have $B(\Omega_1,T,\lambda)\leq \rho \cdot B(\Omega_2,T,\lambda)$.
		\end{enumerate}
		
	\end{lemma}
	\begin{proof}
		For simplicity, let us denote $m_{i,j} := S(\Omega_i, \Omega_j), ~ i=1,2$. The proof of one side of the inequality in $(1)$ is enough since by interchanging the role of $\Omega_1$ and $\Omega_2$, we will get the other inequality.  Consider the analytic function $f:\Omega_1 \to X$  with the series expansion \eqref{E3.1.3}. 
		
		\medskip
		
		Fix $\epsilon_1>0$ and for each multi-index $\alpha$, define
		\begin{equation*}
			c_\alpha := \dfrac{a_\alpha}{(m_{2,1}+\epsilon_1)^{|\alpha|}}.
		\end{equation*}
		Let $g:\Omega_2 \to X$ be a function defined by the series expansion 
		\begin{equation*}
			g(z) := \sum_{\alpha \in \mathbb{N}_0^n} c_\alpha z^\alpha =  \sum_{\alpha \in \mathbb{N}_0^n} a_\alpha \left(\dfrac{z}{m_{2,1}+\epsilon_1}\right)^\alpha.
		\end{equation*}
		By the definition of $m_{2,1}$, we have $\Omega_2 \subset \beta \cdot \Omega_1$, for any $\beta \geq m_{2,1}$ and so we obtain $\Omega_2 \subset (m_{2,1}+\epsilon_1) \cdot \Omega_1$.  In other words, for any $z \in \Omega_2$, we get $z/(m_{2,1}+\epsilon_1) \in \Omega_1$ and as a result $g$ is analytic on $\Omega_2$. 
		
		\medskip
		
		Now fix $0<\epsilon_2<B(\Omega_2,T,\lambda)$ and let $r= B(\Omega_2,T,\lambda)-\epsilon_2$. Then we have
		
		\begin{equation}\label{E3.2.8}
			\sum_{\alpha \in \mathbb{N}_0^n} \sup_{z \in r \cdot \Omega_2} \|T(c_\alpha z^\alpha)\|_Y \leq \lambda \sup_{z \in \Omega_2}\|g(z)\|_X. 
		\end{equation} 
		Now, consider the sets defined by 
		\begin{equation*}
			A_1:= \left\{ \left\|\sum_{\alpha \in {\mathbb{N}_0}^n} c_\alpha z^\alpha \right\|_X : z \in (m_{2,1}+\epsilon_1)\cdot \Omega_1\right \} ~\text{and} ~  A_2:= \left\{ \left\|\sum_{\alpha \in {\mathbb{N}_0}^n} c_\alpha z^\alpha \right\|_X : z \in \Omega_2\right \}.
		\end{equation*}
		The definition of $m_{2,1}$ implies $A_2 \subseteq A_1$ and hence $\sup A_2 \leq \sup A_1$. Now by simple norm estimates, we obtain that 
		\begin{align*}
			\sup_{z \in \Omega_2}\|g(z)\|_X =\sup_{z \in \Omega_2} \left\|\sum_{\alpha \in {\mathbb{N}_0}^n} c_\alpha z^\alpha \right\|_X &\leq \sup_{z \in (m_{2,1}+\epsilon_1) \cdot \Omega_1}  \left\|\sum_{\alpha \in {\mathbb{N}_0}^n} c_\alpha z^\alpha\right\|_X \\&= \sup_{z \in \Omega_1} \left\|\sum_{\alpha \in {\mathbb{N}_0}^n} c_\alpha [(m_{2,1}+\epsilon_1)z]^\alpha\right\|_X
			= \sup_{z \in \Omega_1}\|f(z)\|_X.
		\end{align*}
		Then the inequality \eqref{E3.2.8} will reduce to 
		\begin{equation}\label{E3.2.10}
			\sum_{\alpha \in \mathbb{N}_0^n} \sup_{z \in r\Omega_2} \left\|T(a_\alpha) \left(\dfrac{z}{m_{2,1}+\epsilon_1}\right)^\alpha\right\|_Y \leq \lambda \sup_{z \in \Omega_1}\|f(z)\|_X.
		\end{equation}
		Now fix $\epsilon_3>0$. Since $\Omega_1 \subseteq (m_{1,2}+\epsilon_3) \cdot \Omega_2$, we have the following implication :
		\begin{equation*}
			z \in r \cdot \Omega_1 \implies \dfrac{z}{m_{1,2}+\epsilon_3} \in r \cdot \Omega_2.
		\end{equation*}
		Let us consider the sets defined by 
		\begin{align*}
			A_3&:= \left\{  \left\|T(a_\alpha) \left(\dfrac{z}{(m_{2,1}+\epsilon_1)(m_{1,2}+\epsilon_3)}\right)^\alpha\right\|_Y:z \in r \cdot \Omega_1\right \}, \\ ~\text{and }~  A_4&:= \left\{  \left\|T(a_\alpha) \left(\dfrac{z}{m_{2,1}+\epsilon_1}\right)^\alpha\right\|_Y :z \in r \cdot \Omega_2\right \}.
		\end{align*}
		By the definition of $m_{1,2}$, we have $A_3 \subset A_4$ and hence we have $\sup (A_3) \leq \sup(A_4)$. As a result, we can deduce that 
		\begin{align*}
			\sum_{\alpha \in \mathbb{N}_0^n} \sup_{z \in r \cdot \Omega_1} \left\|T(a_\alpha) \left(\dfrac{z}{(m_{2,1}+\epsilon_1)(m_{1,2}+\epsilon_3)}\right)^\alpha\right\|_Y & \leq  \sum_{\alpha \in \mathbb{N}_0^n} \sup_{z \in r \cdot \Omega_2} \left\|T(a_\alpha) \left(\dfrac{z}{m_{2,1}+\epsilon_1}\right)^\alpha\right\|_Y \\&\leq \lambda \sup_{z \in \Omega_1}\|f(z)\|_X,
		\end{align*}
		where the last inequality is due to \eqref{E3.2.10}. This implies that
		$$\sum_{\alpha \in \mathbb{N}_0^n} \sup_{z \in \delta \cdot \Omega_1} \left\|T(a_\alpha z ^\alpha) \right\|_Y \leq \lambda \sup_{z \in \Omega_1}\|f(z)\|_X,$$
		where 
		$$\delta = \dfrac{B(\Omega_1,T,\lambda) - \epsilon_2}{(m_{1,2}+e_3)(m_{2,1}+\epsilon_1)}.$$
		Since $\epsilon_i, ~ i=1,2,3$ are arbitrary, we conclude that
		$$B(\Omega_1,T,\lambda)\geq \dfrac{B(\Omega_2,T,\lambda)}{m_{1,2}\cdot m_{2,1}}.$$ 
		
		\medskip
		
		\noindent For the proof of $(2)$, note that for any bounded complete Reinhardt domain $\Omega$ of $\mathbb{C}^n$ and $\rho>0$, we have
		$$S(\Omega, \rho \cdot \Omega) = \dfrac{1}{\rho} ~ \text{and} ~ S(\rho\cdot \Omega,  \Omega) =\rho, $$
		which we put in the inequality $(1)$, we will obtain the desired equality. The proof of $(3)$ is obvious from the fact for $\Omega_2 \subset \Omega_1 \subset \rho \cdot \Omega_2, ~ \rho>0$, we have
		$$S(\Omega_1, \Omega_2) \leq \rho ~ \text{and} ~ S(\Omega_2,\Omega_1) \leq 1.$$
		Hence, we completed the proof.
	\end{proof}
	
	\begin{remark}
		If we take $X=\mathbb{C}, ~ \lambda=1$, and $T$ as the identity operator on $X$ with $|f(z)| \leq 1$, then Lemma \ref{L3.2.2} reduces to Lemma 3.1 of\cite{DE04}. Also note that all three results in Lemma \ref{L3.2.2} are valid for $B_m(\Omega,T,\lambda)$, for any $m \in \mathbb{N}$.
	\end{remark}
	
	
	In the following lemma, we provide an upper estimate for the $\lambda$-second Bohr radius $B(B_{W_n}, T, \lambda)$.  If we take $\lambda=1$ and $X=\mathbb{C}$ with $|f(z)|<1$, we get Theorem 4.2 of \cite{DE03}. Though we are using a similar approach as in the proof of Theorem 4.2 of \cite{DE03}, we will provide the proof for the sake of completeness. The result is presented below. 
	\begin{lemma}\label{L3.5.1}
		Let  $W_n=(\mathbb{C}^n, \|\cdot\|)$ be a Banach space such that $\chi((e_k))=1$, where $(e_k)$ denote the canonical basis for $W_n$. Then for each $n \in \mathbb{N}$, we have
		$$B(B_{W_n}, X, \lambda)\leq \dfrac{ \lambda e^3 2^{3/2} \sqrt{\log n}\sup _{\|z\| <1 }\|z\|_2}{\sup _{\|z\| <1 }\|z\|_1}.$$
	\end{lemma}
	\begin{proof}
		Consider an $m$-homogeneous polynomial $P(z) = \sum_{|\alpha|=m}a_\alpha z^\alpha$. Then for any $\zeta_\alpha \in \mathbb{C}$ with $|\zeta_\alpha| \leq 1$, we have
		\begin{align*}
			\left \|\sum_{|\alpha|=m}\zeta_\alpha a_\alpha z^\alpha \right\|_{\mathcal{P}(^m W_n)} \leq  \sup _{\|z\|<1} \sum_{|\alpha|=m} \|\zeta_\alpha a_\alpha z^\alpha \|
			& \leq \sup _{\|z\|<1}  \sum_{|\alpha|=m} \| a_\alpha z^\alpha \|\\
			&\leq \sum_{|\alpha|=m} \sup _{\|z\|<1}\| a_\alpha z^\alpha \|\\
			& \leq \dfrac{\lambda}{B_m(B_{W_n},X,\lambda)^m}  \left \|\sum_{|\alpha|=m}a_\alpha z^\alpha \right\|_{\mathcal{P}(^m W_n)},
		\end{align*}
		where the last inequality is due to the fact that
		$$B_m(B_{W_n},X,\lambda)^m \sum_{|\alpha|=m} \sup _{\|z\|<1}\| a_\alpha z^\alpha \| \leq \lambda \left\|\sum_{|\alpha|=m}a_\alpha z^\alpha \right\|_{\mathcal{P}(^m X)}.$$
		By the definition of $\chi_M(\mathcal{P}(^m W_n))$, we have
		\begin{equation*}
			\chi_M(\mathcal{P}(^m W_n)) \leq \dfrac{\lambda}{B_m(B_{W_n},X,\lambda)^m}.
		\end{equation*}
		Since  for each $m \in \mathbb{N}$, $B(B_{W_n},X,\lambda)^m \leq B_m(B_{W_n},X,\lambda)^m$, we get 
		\begin{equation*}
			B(B_{W_n}, X, \lambda)^m \leq \dfrac{\lambda}{\chi_M(\mathcal{P}(^m W_n))}. 
		\end{equation*}
		Using the same ideas as in the proof of Lemma 4.1 of \cite{DE03}, for each $m \in \mathbb{N}$, we can deduce that 
		\begin{align*}
			B(B_{W_n}, X, \lambda) \leq \left(\lambda \sqrt{m! \log n} 2^{(3m-1)/2} m^{3/2} \dfrac{ \sup _{\|z\| <1 }\|z\|_1}{\sup _{\|z\| <1 }\|z\|_2}\right)^{1/m} \left( \dfrac{ \sup _{\|z\| <1 }\|z\|_2}{\sup _{\|z\| <1 }\|z\|_1}\right).
		\end{align*}
		By taking $m=1$ for the case $n=2$ and  $m=[\log n]$ for $n>2$, we have the inequality from the proof of Theorem 4.2 of \cite{DE03},
		$$\left(\lambda \sqrt{m! \log n} 2^{(3m-1)/2} m^{3/2} \dfrac{ \sup _{\|z\| <1 }\|z\|_1}{\sup _{\|z\| <1 }\|z\|_2}\right)^{1/m} \leq \lambda^{1/m} e^3 2^{3/2} \sqrt{\log n} \leq \lambda e^3 2^{3/2} \sqrt{\log n},$$
		which gives our desired upper estimate. Hence we completed the proof.
	\end{proof}
	
	\section{Proof of the Main Results}
	\subsection{\bf Proof of Theorem \ref{3Tm1}:} Let us denote $(e_k)$ be the standard basis vectors for $B_p^m$. If $p=\infty$, then consider the function
	\begin{equation*}
		f(z):= e_1+e_2z_2=(1,z_2,0,\dots,0), ~ z=(z_1,\dots,z_n) \in \mathbb{D}^n.
	\end{equation*}
	Observe that $\|f(z)\|_\infty= 1 $ and therefore $\sup_{z \in \mathbb{D}^n}\|f(z)\|_\infty =1$. But note that for any $0 < r <1$ we have
	\begin{equation*}
		\sup_{z \in r\mathbb{D}^n} \|e_1\|+ \sup_{z \in r\mathbb{D}^n} \|e_2z_2\|=1+r>1=\sup_{z \in \mathbb{D}^n}\|f(z)\|_\infty.
	\end{equation*}Hence $B(\mathbb{D}^n,\overline{B_p^m},1) =0$ in this case.
	
	\medskip
	
	For the case $1<p<\infty$, we first observe that for any $x>0$,
	$$x^{1/p} - (x-1)^{1/p} =\dfrac{ c^{1/p-1}}{p}, ~ \text{for some} ~ c \in (x-1,x),$$
	which gives 
	$$\lim_{x \to \infty}( x^{1/p} - (x-1)^{1/p}) = 0.$$
	That is, for any $\epsilon>0$, there exists an $R>0$ such that
	$$|x^{1/p} - (x-1)^{1/p}| <\epsilon, ~ \text{for} ~ |x|>R.$$
	In other words, there exists $\mu \in (0,1)$ such that
	\begin{equation}\label{E3.2.4}
		1 - (1-\mu)^{1/p} <\epsilon \mu^{1/p}.   
	\end{equation}
	Now consider the function $f$ defined by
	\begin{align*}
		f(z) := (1-\mu)^{1/p}e_1 + \mu^{1/p} e_2 z_2. 
	\end{align*}
	Simple computations give  
	$$\|f(z)\|_p = (1-\mu)+\mu |z_2| = 1-\mu+\mu  <1,$$ 
	and as a result, we get $\sup_{z \in \mathbb{D}^n}\|f(z)\|_p =1.$ But observe that
	$$(1-\mu)^{1/p}\|e_1\|_p+ \sup_{z \in \epsilon \mathbb{D}^n}\mu^{1/p} \|e_2 z_2\|_p = (1-\mu)^{1/p}+ \epsilon\mu^{1/p}>1 =\sup_{z \in \mathbb{D}^n}\|f(z)\|_p ,$$
	where the last inequality is from \eqref{E3.2.4}. This gives $B(\mathbb{D}^n,\overline{B_p^m},1) =0$.
	
	\medskip
	Now consider the case $p=1$ and let $\epsilon>0$. Similarly as in \eqref{E3.2.4}, there exists $\eta \in (0,1)$ such that
	\begin{equation}\label{E3.2.5}
		1-\sqrt{1-\eta} < \epsilon \sqrt{\eta}.
	\end{equation}
	Let us define the function 
	\begin{align}\label{E3.2.6}
		\notag f(z) &:= \left(\dfrac{\sqrt{1-\eta}+z_2\sqrt{\eta}}{2}, \dfrac{\sqrt{1-\eta}-z_2\sqrt{\eta}}{2},0,\dots, 0\right) \\&=  \dfrac{\sqrt{1-\eta}}{2}(1,1,0,\dots,0)+ \dfrac{\sqrt{\eta}}{2}(1,-1,0,\dots,0)z_2.
	\end{align}
	Note that in the definition of the function $f$ given by \eqref{E3.2.6}, the number $0$ is counted $m-2$ times.
	An easy observation shows that
	\begin{align*}
		\|f(z)\|_1 &=\dfrac{|\sqrt{1-\eta}+z_2\sqrt{\eta}|+|\sqrt{1-\eta}-z_2\sqrt{\eta}|}{2}\\ &\leq  \left(\dfrac{|\sqrt{1-\eta}+z_2\sqrt{\eta}|^2+|\sqrt{1-\eta}-z_2\sqrt{\eta}|^2}{2}\right)^{1/2} = 1-\eta+\eta |z_2|^2 <1,
	\end{align*}
	and so we have $\sup_{z \in \mathbb{D}^n}\|f(z)\|_1<1.$ On the other hand, using \eqref{E3.2.5} we get
	\begin{align*}
		\dfrac{\sqrt{1-\eta}}{2}\|(1,1,0,\dots,0)\|_1+ \sup_{z \in \epsilon\mathbb{D}^n}\dfrac{\sqrt{\eta}}{2}\|(1,-1,0,\dots,0)\|_1 |z_2|= \sqrt{1-\eta}+\epsilon \sqrt{\eta}>1.
	\end{align*}
	Comparing the lower bound with $\sup_{z \in \mathbb{D}^n}\|f(z)\|_1$, we conclude that $B(\mathbb{D}^n,\overline{B_p^m},1) =0$. Hence, the proof is completed.
	\hfill{$\Box$}
	
	\subsection{\bf Proof of Theorem \ref{3T1}}  For any $z \in \mathbb{C}^n$, we have 
	$$\|z\|_\infty \leq \|z\|_p ~ \text{and} ~ \|z\|_p \leq n^{1/p} \|z\|_\infty,$$
	which gives
	\begin{equation}\label{E3.2.11}
		S(\mathbb{D}^n, B_p^n) \leq n^{1/p} ~ \text{and} ~ S(B_p^n, \mathbb{D}^n)\leq 1.
	\end{equation}
	Considering the vectors $z=(1, \dots,1)$ and  $z=(1,0, \dots, 0)$ respectively, equality happens in both quantities in \eqref{E3.2.11}.
	
	\medskip
	By putting the values of $S(\mathbb{D}^n, B_p^n)$ and $S(B_p^n, \mathbb{D}^n)$ in $(1)$ of Lemma \ref{L3.2.2}, we obtain the following.
	\begin{align}\label{E3.2.12}
		B(B_p^n,T,\lambda) \geq \dfrac{B(\mathbb{D}^n,T,\lambda)}{n^{1/p}} 
		=  \dfrac{K(\mathbb{D}^n,T,\lambda)}{n^{1/p}}, 
	\end{align}
	where the last equality is due to Proposition \ref{3P1}. Now, apply the lower bound for $K(\mathbb{D}^n,T,\lambda)$ from Proposition 3.3 of \cite{DE12} to the inequality \eqref{E3.2.12}, and we obtain the desired upper bound. Also note that if $\|T\| < \lambda$, then we have $\rho, \sigma_1$ and $\sigma_2$ are all strictly positive, and as a result, we conclude that $B(B_p^n,T,\lambda)>0$.	\hfill{$\Box$}
	
	\subsection{\bf Proof of Theorem \ref{LOGTHM}.} 
	Fix a  $\psi \in X^*$ with $\|\psi\|<1$ and $z \in B_p^n$. Now define the function $g:\mathbb{D} \to \mathbb{C}$ given by
	\begin{equation*}
		g(w):= \psi(f(zw))=\psi\Bigg(\sum_{\alpha \in \mathbb{N}_0^n} a_\alpha z^\alpha w^{|\alpha|}\Bigg)=\psi(a_0)+\sum_{k=1}^\infty \Bigg(\sum_{|\alpha|=k} \psi(a_\alpha)z^\alpha\Bigg) w^k.
	\end{equation*}
	It is clear that $|g(w)|=|\psi(f(zw))| \leq \|\psi\|\|f(zw\|<1$. Now applying the Weiner inequality for the function $g$, we have that
	\begin{equation}\label{HBINEQ}
		\Bigg | \psi \Bigg(\sum_{|\alpha|=k} a_\alpha z^\alpha \Bigg)\Bigg|=\Bigg | \sum_{|\alpha|=k} \psi(a_\alpha)z^\alpha \Bigg| \leq 1-|\psi(a_0)|^2.
	\end{equation}
	Since \eqref{HBINEQ} is valid for any $\psi \in X^*$ such that $\|\psi\|<1$ and any choice of  $z \in B_p^n$, as a consequence of the Hahn-Banach theorem, for any $k \in \mathbb{N}$, we have 
	\begin{equation}\label{WEINER}
		\sup _{z \in B_p^n} \Bigg \|  \sum_{|\alpha|=k} a_\alpha z^\alpha \Bigg\| \leq 1-\|a_0\|^2.
	\end{equation}
	Now for a fixed $\psi \in X^*$ with $\|\psi\|<1$, choose $\zeta_\alpha \in \overline{\mathbb{D}}$ such that $\zeta_\alpha \psi(a_\alpha)=|\psi(a_\alpha)|$. Now, for any $m \in \mathbb{N}$, we deduce that
	\begin{align}\label{INEQFINAL}
		\nonumber   \sum_{|\alpha|=m}\|a_\alpha\| \dfrac{1}{n^{m/p}} =  \sum_{|\alpha|=m}\|a_\alpha\| \Bigg(\dfrac{1}{n^{1/p}}\Bigg)^\alpha &=  \sum_{|\alpha|=m} \sup_{\|\psi\|<1}|\psi(a_\alpha)| \Bigg(\dfrac{1}{n^{1/p}}\Bigg)^\alpha\\
		\nonumber   &= \sum_{|\alpha|=m} \sup_{\|\psi\|<1}\zeta_\alpha \psi(a_\alpha) \Bigg(\dfrac{1}{n^{1/p}}\Bigg)^\alpha\\
		\nonumber    &= \sup_{\|\psi\|<1} \psi \Bigg(\sum_{|\alpha|=m} \zeta_\alpha a_\alpha \Bigg(\dfrac{1}{n^{1/p}}\Bigg)^\alpha\Bigg)\\
		\nonumber    & \leq \sup_{z \in B_p^n}\Bigg\|\sum_{|\alpha|=m} a_\alpha\zeta_\alpha z^\alpha\Bigg\|\\
		&\leq \chi_M(\mathcal{P}(^m l_p^n)\sup_{z \in B_p^n}\Bigg\|\sum_{|\alpha|=m} a_\alpha z^\alpha\Bigg\|
	\end{align}
	From the definition of $K(B_p^n,T,\lambda)$, it is clear that
	\begin{align*}
		\left \|\sum_{|\alpha|=m}T(\zeta_\alpha a_\alpha) z^\alpha \right\| \leq  \sup _{\|z\|_p<1} \sum_{|\alpha|=m} \|\zeta_\alpha T( a_\alpha) z^\alpha \|
		& \leq \sup _{\|z\|_p<1}  \sum_{|\alpha|=m} \| T(a_\alpha) z^\alpha \|\\
		&\leq \dfrac{\lambda}{K(B_P^n,T,\lambda)^m}  \left \|\sum_{|\alpha|=m}a_\alpha z^\alpha \right\|,
	\end{align*}
	which gives 
	\begin{equation}\label{chi}
		\chi_M(\mathcal{P}(^m l_p^n) \leq \dfrac{\lambda}{K(B_P^n,T,\lambda)^m}. 
	\end{equation}
	Applying \eqref{chi} and \eqref{WEINER} to \eqref{INEQFINAL}, we get
	\begin{equation}\label{NORMINEQ}
		\sum_{|\alpha|=m}\|a_\alpha\| \leq \dfrac{\lambda (1-\|a_0\|^2)n^{m/p}}{K(B_P^n,T,\lambda)^m}.
	\end{equation}
	As a consequence of Lemma 3.5 of \cite{DE04}, for any $ \alpha ~\text{such that } ~ |\alpha|=m$, we obtain that
	\begin{equation}\label{E3.4.20}
		S_{p,\alpha}:=\sup_{\|z\|_p<1} |z^\alpha| = \left(\dfrac{\alpha^\alpha}{|\alpha|^{|\alpha|}}\right)^{1/p}\geq \dfrac{1}{m^{m/p}}. 
	\end{equation}As a result, for any $r \in [0,1)$, we have
	\begin{align}\label{NORRMINEQ}
		\nonumber  \sum_{\alpha \in \mathbb{N}_0^n}\sup_{\|z\|_p<r} \|T(a_\alpha) z^\alpha\|&=\|T(a_0)\|+\sum_{k=1}^\infty \sum_{|\alpha|=k}\|T(a_\alpha)\| \sup_{\|z\|_p<1}|(rz)^\alpha|\\
		\nonumber    &\leq \lambda\Bigg(\|a_0\|+\sum_{k=1}^\infty r^k\sum_{|\alpha|=k}\|a_\alpha\| \sup_{\|z\|_p<1}|z^\alpha|\Bigg)\\
		&=\lambda\Bigg(\|a_0\|+\sum_{k=1}^\infty r^k\sum_{|\alpha|=k}\|a_\alpha\| \Bigg(\dfrac{\alpha^\alpha}{m^m}\Bigg)^{1/p}\Bigg).
	\end{align}
	Since for any $\alpha \in \mathbb{N}_0^n$ with $|\alpha|=m$,  $\alpha^\alpha \leq |\alpha|^{|\alpha|}=m^m$, and applying \eqref{NORMINEQ} to \eqref{NORRMINEQ},  we get 
	\begin{align*}
		\sum_{\alpha \in \mathbb{N}_0^n}\sup_{\|z\|_p<r} \|T(a_\alpha) z^\alpha\|&\leq \lambda \Bigg(\|a_0\|+\sum_{k=1}^\infty r^k\sum_{|\alpha|=k}\|a_\alpha\|\Bigg)\\
		&\leq \lambda \Bigg(\|a_0\|+\lambda (1-\|a_0\|^2)\sum_{k=1}^\infty \Bigg(\dfrac{rn^{1/p}}{K(B_P^n,X,\lambda)}\Bigg)^k\Bigg).
	\end{align*}
	Notice that if $r$ satisfies the inequality
	\begin{equation*}
		r \leq \dfrac{K(B_p^n,T,\lambda)}{(\lambda(1+\|a_0\|)+1)n^{1/p}},
	\end{equation*}
	we get 
	\begin{equation*}
		\sum_{k=1}^\infty \Bigg(\dfrac{rn^{1/p}}{K(B_P^n,X,\lambda)}\Bigg)^k \leq \sum_{k=1}^\infty \Bigg(\dfrac{1}{\lambda(1+\|a_0\|)+1} \Bigg)^k=\dfrac{1}{\lambda(1+\|a_0\|)},
	\end{equation*}
	which gives
	\begin{align*}
		\sum_{\alpha \in \mathbb{N}_0^n}\sup_{\|z\|_p<r} \|T(a_\alpha) z^\alpha\|&\leq \lambda\Bigg(\|a_0\|+\dfrac{1-\|a_0\|^2}{1+\|a_0\|}\Bigg)=\lambda=\lambda \sup_{\|z\|_p<1} \|f(z)\|_X.
	\end{align*}
	This shows that 
	\begin{equation*}
		B(B_p^n,T,\lambda) \geq \dfrac{K(B_p^n,T,\lambda)}{(\lambda(1+\|a_0\|)+1)n^{1/p}}\geq \dfrac{K(B_p^n,T,\lambda)}{(1+2\lambda)n^{1/p}}.
	\end{equation*}
	Now applying Theorem 3.6 of \cite{SH23} to $K(B_p^n,T,\lambda)$, there exists $C_0>0$ such that
	\begin{equation*}
		B(B_p^n,T,\lambda) \geq \dfrac{C_0}{1+2\lambda} \Bigg(\dfrac{\|T\|-\lambda}{\|T\|-2\lambda}\Bigg) \dfrac{(\log n)^{1-\frac{1}{q}}}{n^{\frac{pq-p+q}{pq}}}.
	\end{equation*}
	Hence, the proof is completed.

	\subsection{\bf Proof of Theorem \ref{THM2.12}} 
	Given that $T$ is the identity operator on $X$, where $X$ is a finite-dimensional Banach space. Applying  Theorem 4.1 of \cite{DE12} to the inequality \eqref{E3.2.12}, there exists a constant $F_X$ such that
	\begin{align*}\label{E3.2.13}
		B(B_p^n,X,\lambda) &\geq 
		\dfrac{K(\mathbb{D}^n,X,\lambda)}{n^{1/p}} \geq \dfrac{F_X (\lambda-1)  \sqrt{\log n}}{(2 \lambda-1)(n^{1/p+1/2})}, 
	\end{align*}
	where $K(\mathbb{D}^n,X,\lambda)$ is the vector-valued Bohr radius for analytic functions defined on $\mathbb{D}^n$ under the case where $T$ is the identity operator on $X$. This gives the required lower bound.
	
	From Lemma \ref{L3.2.2}, we have
	$$B(B_p^n,X,\lambda) \leq n^{1/p} B(\mathbb{D}^n,X,\lambda)= n^{1/p} K(\mathbb{D}^n,X,\lambda).$$
	Similar to the case of lower estimates, use the upper bound for $K(\mathbb{D}^n,T,\lambda)$ which is given in \cite[Theorem 4.1]{DE12} to obtain the desired upper bound.
	\hfill{$\Box$}
	\subsection{Proof of Theorem \ref{T3.4.3}} If $B(B_p^n,X,\lambda)=0$, nothing to prove. So we omit this case. Now, let $(e_k)$ be the Schauder basis for $X$ with $\|e_k\|_X=1$. For a fixed $m \in \mathbb{N}$ consider the function $f : B_p^n \to X$, which is defined by
	$$f(z) := g(z) e_1 := \sum_{|\alpha|=m} a_\alpha z^\alpha,$$
	where $g:B_p^n \to \mathbb{C}$ be the polynomial  defined by 
	\begin{equation*}
		g(z) := \sum_{|\alpha|=m} \dfrac{m!}{\alpha !} z^\alpha.
	\end{equation*}
	Obviously we have $$a_\alpha = \dfrac{m! e_1}{\alpha!} ~ \text{and} ~ \|f(z)\|_X = |g(z)|.$$
	Take $0<\epsilon<B(B_p^n,X,\lambda)$ and define $r= B(B_p^n,X,\lambda)-\epsilon$. It is to be noted that since the ball $B_p^n$ is a complete Reinhardt domain, from the proof of Lemma 2.1 of \cite{DE04}, we deduce that
	\begin{align*}
		\sup_{\|z\|_p<1} \left\| \sum_{|\alpha|=m} \zeta_\alpha a_\alpha z^\alpha\right\|_X  &=\sup_{\|z\|_p<1} \left\| \sum_{|\alpha|=m} \zeta_\alpha\dfrac{m!e_1}{\alpha !} z^\alpha\right\|_X\\&=\sup_{\|z\|_p<1} \left|\sum_{|\alpha|=m}\dfrac{m!}{\alpha!}\zeta_\alpha z^\alpha \right|= \sup_{\|z\|_p<1} \sum_{|\alpha|=m} \left|\dfrac{m!}{\alpha !}\zeta_\alpha z^\alpha\right|,
	\end{align*}
	where  for each  $\alpha~ \text{with} ~ |\alpha|=m$, $(\zeta_\alpha)$ is an independent standard Bernoulli random variable as in the proof of Theorem 3.3 of \cite{DE04}.
	This gives the inequality
	\begin{align}\label{E3.4.2}
		\notag  \sum_{|\alpha|=m} \sup _{\|z\|_p<r} \|T(a_\alpha) \zeta_\alpha z^\alpha\|_X &\leq \lambda \sup_{\|z\|_p<1}\left\|\sum_{|\alpha|=m}a_\alpha\zeta_\alpha z^\alpha \right\|_X\\ 
		&= \lambda \sup_{\|z\|_p<1} \left|\sum_{|\alpha|=m}\dfrac{m!}{\alpha!}\zeta_\alpha z^\alpha \right|.
	\end{align}
	
	\noindent But note that 
	\begin{align*}
		\sum_{|\alpha|=m} \sup _{\|z\|_p<r} \|T(a_\alpha)\zeta_\alpha z^\alpha\|_X &=\sum_{|\alpha|=m} r^m \sup _{\|z\|_p<1} \left\|T\left(\dfrac{m! e_1}{\alpha!}\right)\zeta_\alpha z^\alpha \right\|_X \\
		&=\sum_{|\alpha|=m} \dfrac{r^m m!}{\alpha!} \|T(e_1)\|_X \sup _{\|z\|_p<1}|z^\alpha| \\
		&= \sum_{|\alpha|=m} \dfrac{r^m m!}{\alpha!}  \sup _{\|z\|_p<1}|z^\alpha|.
	\end{align*}
	Then by  inequality \eqref{E3.4.2}, we have
	\begin{equation}\label{E3.4.7}
		\sum_{|\alpha|=m} \dfrac{r^m m!}{\alpha!}  \sup _{\|z\|_p<1}|z^\alpha| \leq \lambda \sup_{\|z\|_p<1} \left|\sum_{|\alpha|=m} \dfrac{m!}{\alpha!} \zeta_\alpha z^\alpha \right|. 
	\end{equation}
	By the multinomial expansion, we have 
	\begin{equation} \label{E3.4.21}
		\sum_{|\alpha|=m} \dfrac{m!}{\alpha!} =n^m.
	\end{equation}
	Combining \eqref{E3.4.20} and \eqref{E3.4.21} along with \eqref{E3.4.20}, we get 
	\begin{align*}
		\sum_{|\alpha|=m} \dfrac{r^m m!}{\alpha!}  \sup _{\|z\|_p<1}|z^\alpha| &= r^m  \sum_{|\alpha|=m} \dfrac{m!}{\alpha!}\left(\dfrac{\alpha^\alpha}{|\alpha|^{|\alpha|}}\right)^{1/p} \\
		&\geq \left(\dfrac{r}{m^{1/p}}\right)^m \sum_{|\alpha|=m}\dfrac{m!}{\alpha!}\\
		&=  \left(\dfrac{rn}{m^{1/p}}\right)^m.
	\end{align*}
	Then \eqref{E3.4.7} will reduce to
	\begin{equation}\label{E3.4.10}
		\left(\dfrac{rn}{m^{1/p}}\right)^m \leq \lambda \sup_{\|z\|_p<1} \left|\sum_{|\alpha|=m} \dfrac{m!}{\alpha!} \zeta_\alpha z^\alpha \right| \leq \lambda \sup_{\|z\|_p<1} \left|\sum_{|\alpha|=m} \dfrac{m!}{\alpha!}  z^\alpha \right|.
	\end{equation}
	As in the proof of Theorem 3.3 of \cite{DE04}, we have
	\begin{equation*}
		\sup_{\|z\|_p<1} \left|\sum_{|\alpha|=m} \dfrac{m!}{\alpha!} \zeta_\alpha z^\alpha \right| \leq m^{3/2} 2^{\frac{3m-1}{2}} \sqrt{\log n} \sup_{|\alpha|=m} \left(|a_\alpha| \sqrt{\alpha!/m!}\right) \sup_{\|z\|_p<1} \|z\|_2^{m-1} \sup_{\|z\|_p<1} \|z\|_1,
	\end{equation*}
	from which we can deduce that
	
	\begin{align*}
		(rn)^m &\leq \lambda m^{3/2} m^{1/p} 2^{\frac{3m-1}{2}} \sqrt{\log n \cdot (m!)} \sup_{\|z\|_p<1} \|z\|_2^{m-1} \sup_{\|z\|_p<1} \|z\|_1.
	\end{align*}
	If we take the $m$-th root, then we obtain that
	\begin{align}\label{E3.4.1}
		rn \leq  \left( \lambda m^{3/2} m^{1/p} 2^{\frac{3m-1}{2}} \sqrt{\log n \cdot (m!)}\right )^{1/m} \left(\dfrac{ \sup_{\|z\|_p<1} \|z\|_1}{\sup_{\|z\|_p<1} \|z\|_2}\right)^{1/m} \sup_{\|z\|_p<1} \|z\|_2.
	\end{align}
	Since for any vector $z=(z_1,\dots,z_n)$, we have the inequality
	\begin{equation*}
		\|z\|_\infty \leq \|z\|_1 \leq n\|z\|_\infty,
	\end{equation*}
	we deduce that 
	\begin{equation*}
		\left(\dfrac{ \sup_{\|z\|_p<t} \|z\|_1}{\sup_{\|z\|_p<t} \|z\|_2}\right)^{1/m} \leq  \left(\dfrac{ n\sup_{\|z\|_p<t} \|z\|_\infty}{\sup_{\|z\|_p<t} \|z\|_\infty}\right)^{1/m}. 
	\end{equation*}
	As a result, we can find a $t>0$ such that 
	\begin{align*}
		\left(\dfrac{ \sup_{\|z\|_p<1} \|z\|_1}{\sup_{\|z\|_p<1} \|z\|_2}\right)^{1/m} &= \left(\dfrac{ \sup_{\|z\|_p<t} \|z\|_1}{\sup_{\|z\|_p<t} \|z\|_2}\right)^{1/m} \leq n^{1/m}.
	\end{align*}
	Then the inequality \eqref{E3.4.1} will reduces to
	\begin{equation}\label{E3.4.3}
		rn \leq \left( \lambda m^{3/2} m^{1/p} 2^{\frac{3m-1}{2}} \sqrt{m!} (\log n)^{1/p+1/2} \cdot n\right )^{1/m} \sup_{\|z\|_p<1} \|z\|_2.
	\end{equation}
	Also from Theorem 4.2 of \cite{DE03}, for each $n\geq 2$, there exists an $m$ such that 
	\begin{equation}\label{E3.4.4}
		\left(m^{3/2} m^{1/p}2^{\frac{3m-1}{2}}\sqrt{m!}\sqrt{\log n} \cdot n\right)^{1/m} < e^3 2^{3/2} \sqrt{\log n}.
	\end{equation}
	Applying  \eqref{E3.4.4} to \eqref{E3.4.3} along with the fact that $\lambda^{1/m} \leq \lambda$, we obtain that
	$$rn \leq \lambda (\log n)^{1/p}  e^3 2^{3/2} \sqrt{\log n} \sup_{\|z\|_p<1} \|z\|_2.$$
	Since
	\begin{equation*}
		\sup_{\|z\|_p<1} \|z\|_2 \leq \begin{cases}
			1, &~ \text{if}~ 1 \leq p \leq 2,\\
			n^{\frac{1}{2}-\frac{1}{p}}, &~ \text{if} ~ 2 \leq p < \infty,
		\end{cases}
	\end{equation*}
	we have the required upper bound. Hence, the proof is completed.
	\hfill{$\Box$}
	
	\subsection{\bf Proof of Theorem \ref{THM2.15}} We use the similar idea of Corollary 5.4 of \cite{DE03}. By Lemma \ref{L3.5.1}, we have
	\begin{align*}
		B(B_{W_n}, X, \lambda)  B(B_{W_n^*}, X, \lambda)\leq (8\lambda^2 e^6  \log n )\dfrac{ \sup _{\|z\| <1 }\|z\|_2 \cdot \sup _{\|z\|^* <1 }\|z\|_2}{\sup _{\|z\| <1 }\|z\|_1 \cdot \sup _{\|z\|^* <1 }\|z\|_1},
	\end{align*}
	where $\|\cdot\|$ and $\| \cdot \|^*$ denote the norms on the spaces $W_n$ and $W_n^*$ respectively.
	It is trivial that
	\begin{equation*}
		\sup _{\|z\| <1 }\|z\|_1 = \|I : W_n \to l^n_1\| ~ \text{and} ~ \sup _{\|z\| <1 }\|z\|_2 = \|I : W_n \to l^n_2\|,
	\end{equation*}
	where $I$ denotes the identity operator. Now we consider the Banach-Mazur distance between to normed spaces $X$ and $Y$, which is defined as
	\begin{equation*}
		d(X,Y) := \inf \{\|T\| \|T^{-1}\|: T ~\text{is an invertible operator from}~ X ~\text{to} ~Y\}.
	\end{equation*}
	Combining the identities (5.5) and (5.6) of \cite{DE03}, we can write
	\begin{equation*}
		B(B_{W_n}, X, \lambda)  B(B_{W_n^*}, X, \lambda)\leq \dfrac{(8\lambda^2 e^6  \log n ) d(W_n,l_2^n)}{n}.
	\end{equation*}
	Use the estimate  $d(W_n,l_2^n) \leq \sqrt{n}$ from \cite[p.249]{TO89} to obtain  
	\begin{equation*}
		B(B_{W_n}, X, \lambda)  B(B_{W_n^*}, X, \lambda)\leq \dfrac{(8\lambda^2 e^6  \log n )}{\sqrt{n}}.
	\end{equation*}
	Since we have $$\lim_{n \to \infty} \dfrac{\log n}{\sqrt{n}}=0,$$ the proof is completed.
	\hfill{$\Box$}

	\bigskip
	\noindent 
	\\
	{\bf Acknowledgment.}  The work of the first author is supported by SERB-SRG (SRG/2023/001938). The work of the second author is supported by an INSPIRE fellowship of the Department of Science and Technology, Govt. of India (DST/INSPIRE Fellowship/2021/IF210612).

	\subsection*{Conflict of Interests}
	The authors declare that there is no conflict of interests regarding the publication of this paper.
	
	\subsection*{Data Availability Statement}
	The authors declare that this research is purely theoretical and does not associate with any data.

\end{document}